\documentclass[10pt,twocolumn,twoside]{ieeetran}     

\IEEEoverridecommandlockouts                              

\usepackage{graphics} 
\usepackage{epsfig} 
\usepackage{mathptmx} 
\usepackage{times} 
\usepackage{amsmath} 
\usepackage{amssymb}  
\usepackage{graphicx}
\usepackage{relsize}
\DeclareMathAlphabet{\mathcal}{OMS}{cmsy}{m}{n}
\SetMathAlphabet{\mathcal}{bold}{OMS}{cmsy}{b}{n}

\title{\LARGE \bf
Inverse optimal control for angle stabilization in \\ converter-based generation}
\author{Taouba Jouini, Anders Rantzer and Emma Tegling 
\thanks{*This work has received funding from the European Research Council (ERC) under the European Union's Horizon 2020 research  and innovation program (grant agreement No: 834142) and the Swedish Research Council (grant 2019-0069).}%
\thanks{\tb This is an extended version of a manuscript accepted for publication at the American Control Conference (ACC) in Atlanta, June 8-10, 2022. In this version, Section IV compares the performance of the angular droop controller to standard frequency droop control.}
\thanks{The authors are with the Department of Automatic Control, LTH, Lund University,
        Naturvetarvägen 18, 223 62, Lund, Sweden. 
        E-mails:
        \tt\small \{taouba.jouini, anders.rantzer, emma.tegling\}@control.lth.se}}%

\graphicspath{{./fig/}}
\usepackage[usenames,dvipsnames]{color}
\usepackage{epstopdf}

\usepackage{amsmath,amsthm,amssymb,mathrsfs,dsfont}
\usepackage{amsfonts}
\usepackage{siunitx}
\usepackage[colorinlistoftodos,prependcaption,textsize=tiny]{todonotes}

\usepackage{tikz}
\usetikzlibrary{matrix}

\usepackage{cite}

\newcommand\oprocendsymbol{\hbox{$\blacksquare$}}

\newcommand{\dd}[0]{\mathrm d}
\newcommand\oprocend{\relax\ifmmode\else\unskip\hfill\fi\oprocendsymbol}

\newcommand{\real}[0]{\mathbb R}

\providecommand{\norm}[1]{\lVert#1\rVert}

\newcommand{\hn}{$\mathcal{H}_2$ }

\newtheorem{theorem}{Theorem}[section]

\newtheorem{definition}[]{Definition} 

\newtheorem{proposition}[theorem]{Proposition}
\newtheorem{result}[theorem]{Result}

\newtheorem{remark}[]{Remark}
\newtheorem{assumption}[]{Assumption}

\usepackage{verbatim} 

\newcommand{\tb}[0]{\color{blue}}

\usepackage[normalem]{ulem}
\newcommand\rout{\bgroup\markoverwith{\textcolor{red}{/}}\ULon} 

\usepackage[prependcaption,colorinlistoftodos]{todonotes}

\begin{document}

\maketitle
\thispagestyle{empty}
\pagestyle{empty}

\begin{abstract}
       In inverse optimal control, the optimality of a given feedback stabilizing controller is a byproduct of the choice of a meaningful, a posteriori defined, cost functional. This allows for a simple tuning comparable to linear quadratic control, also for nonlinear controllers. Our work illustrates the usefulness of this approach in the control of converter-based power systems and networked systems in general, and thereby in finding controllers with topological structure and known optimality properties. In particular, \color{black} we {design} an inverse optimal feedback controller that stabilizes the phase angles of voltage-source controlled {DC/AC} converters at an induced steady state with {\em zero} frequency error. The distributed angular droop controller yields active power to angle droop behavior at steady state.  Moreover, we suggest a practical implementation of the controller and corroborate our results through simulations on a three-converter system {and a numerical comparison with standard frequency droop control}.
\end{abstract}


\section{INTRODUCTION}

{ A} diagnosis of the event of September 28, 2016 in Australia shows anomalous power systems dynamics caused by a series of voltage dips~\cite{operator2017black}. This was originated by the growing angle difference between { the voltage phase angles of two areas in South Australia prior to the separation of South Australia from the remainder of the electrical grid}. Following separation, sudden phase angle changes accompanied by a rapid change in the load have resulted in inaccuracies in short-term frequency measurements~\cite{paolone2020fundamentals}. 
{ A} lesson that can be drawn from the event in Australia is the importance of phase angles in monitoring the stability of converter-based generation and, in particular, in providing useful information that can be exploited for a better design of control schemes for converters~\cite{paolone2020fundamentals}.
Recently, different DC/AC converter control strategies have been proposed to stabilize the output voltage angles at a desired steady state, { f}or example, based on gradient systems and Kuramoto-like oscillator dynamics~\cite{arghir2019electronic, tayyebi2020hybrid}. 
{ Similarly, our work aims to control the angles of DC/AC converters.} 

{ Optimization remains} an important theoretical tool for stability and control in power systems~\cite{molzahn2017survey} and is the backbone of a plethora of strategies for an improved operation of the electrical grid. 
In~\cite{7852234}, dynamic online feedback optimization is used to synthesize controllers, while accounting for input and output constraints and allowing for non-smooth feasible sets based on projected gradient descent algorithms.
Furthermore, the online feedback optimization discussed in~\cite{colombino2019online} enables the study of time-varying convex optimization problems, while allowing for disturbance rejection and exact tracking, and is showcased for power transmission systems to compress the time scales between secondary and tertiary control. Feedback optimization based on dynamic programming is deployed in~\cite{guo2018stochastic} for power scheduling of converters and the associated operational cost in a data-driven stochastic framework. 

In optimal control, it is well-known that every meaningful value function is a Lyapunov function. This constitutes an important link between stability and optimality and allows for the systematic analysis of optimal feedback controllers. 
In {\em inverse} optimal control, the converse link is established. Namely, it is shown that every Lyapunov function is a meaningful value function. This allows for a systematic { design} of feedback controllers associated with control Lyapunov functions, that are optimal with respect to an a {\em posteriori} specified cost functional, satisfying the Hamilton-Jacobi-Bellman (HJB) equation. This was first spotted by R.E. Kalman~\cite{kalman1964linear} for linear systems with quadratic cost and later extended to nonlinear systems by Moylan and Anderson in~\cite{moylan1973nonlinear}, Casti et al.~\cite{casti1980general} for a class of cost functionals that are, e.g., strictly convex in the input for a fixed state, subject to general nonlinear systems. Afterwards, Freeman and Kokotovic incremented the system dynamics { with disturbances and incorporated the constraints} in~\cite{freeman1996inverse} { to study} the inverse robust stabilization problem leading to the analysis of the Hamilton-Jacobi-Isaacs (HJI) equation. 
Our previous work in~\cite{9429728} exploits the same theory { to design a distributed controller} in coupled second-order oscillators.

In this work, we consider a network of voltage-source controlled converters, each of which is equipped with the capability of actuating the voltage phase angle, using synchrophasors. Synchrophasors are time-synchronized electrical measurements that represent both the magnitude and phase angle of the electrical sinusoids, measured by fast time-stamped devices, or { phasor measurement units (PMUs)}, and constitute the basis of real-time monitoring and control actions in the electric grid~\cite{ usman2019applications}. 
In particular, we formulate an inverse optimal control problem, where a {\em distributed} solution to the HJB equation can be found without expensive computations. From a theoretical point of view, the proposed controller demonstrates the usefulness of inverse optimal control theory in networked settings via synthesis of {\em the angular droop control}{, a feat that is otherwise challenging}.  

{ The angular droop controller, designed for the multi-converter system, coincides with that proposed in~\cite{6987381,7419922}. In these works, only a linear stability analysis is conducted and optimality is not established. Here, we prove local asymptotic stability of the induced steady state angle with respect to nonlinear system dynamics.  
The angular droop controller turns out to be the inverse optimal locally stabilizing control law for the multi-converter system with respect to a meaningful cost functional. { As such,} our control design bridges a gap between control theorists and power system experts{ , by demonstrating optimality for the intuitively appealing controller of power converters.}
{ The optimal controller} has desired gradient descent form and possesses} grid-forming capabilities contributing to angle stabilization and thus achieves both primary and secondary frequency control, i.e., zero frequency error. 
Finally, we validate our results on a high-order model of three DC/AC converter system in closed-loop with the angular droop control, give nuts and bolts on how a practical implementation can be achieved { and provide a numerical comparison to standard frequency droop control{~\cite{SIMPSONPORCO20132603}} demonstrating, in particular, improved scalability properties to large networks}.

{\em Notation:} 
For a matrix $P\in\real^{n\times n},\; P=P^\top>0$ and a vector $v\in\real^n$, let $\norm{v}_{P}=\sqrt{v^\top P \,v}$. Let $\mathrm{diag} (v)$ be the diagonal matrix with elements $v_i, \,  i=1,\dots ,n$, $\norm{v}_{\infty}=\sup_{i=1\dots n} \vert v_i\vert$ be the maximum norm of $v$, and $\underline{\sin}(v)$ and $\underline{\cos}(v)$ be the vector-valued sine and cosine functions. Given a { twice} continuously differentiable function $V(x)$, let $\nabla_x V=\frac{\partial V}{\partial x}$ be the the gradient of $V$ with respect to $x$ and $\nabla^2_x V=\frac{\partial^2 V}{\partial^2 x}$ its Hessian matrix. 
{ For $p\in\mathds{N}$, l}et $I_p$ be the $p\times p$ identity matrix and $\mathds{1}_p$ be the $p\times 1$ vector of all ones. Given a dynamical system, $\dot x(t)= f(x(t)),\, x(0)=x_0$, we consider the system to be time-invariant throughout and mostly drop the time-dependence of the state variables in the notation. 

Furthermore, consider a network described by a connected graph~$\mathcal{G}=(\mathcal{V}, \mathcal{E},  \Xi)$, consisting of $\vert\mathcal{V}\vert=n$ nodes representing  DC/AC converter buses and $\vert\mathcal{E}\vert=m$ edges { modeling} purely inductive transmission lines (i.e., with zero conductances) with susceptance $b_{kj}>0,\; (k,j)\in\mathcal{E}$ collected in the diagonal matrix~$\Xi=\mathrm{diag}(b_{kj}),\; (k,j)\in\mathcal{E}$. The topology of the graph~$\mathcal{G}$ is described by the incidence matrix $\mathcal{B}\in\real^{n\times m}$. Let $\mathcal{N}_k$ be the neighbor set of converter $k$. We denote by $\mathcal{L} = \mathcal{B}\, \Xi\, \mathcal{B}^\top$ the bus admittance matrix of $\mathcal{G}$, which is a weighted Laplacian with eigenvalues $0=\lambda_1 < \lambda_2\leq \dots \leq \lambda_n$. 

\section{Problem formulation}

\label{sec: prob-formulation}
In this section, we start by presenting the multi-converter model following~\cite{kundur1994power, dorfler2016breaking} and then formulate the corresponding optimal control problem. This underlies the analysis of the angular feedback control that is at the core of our main result. 
\subsection{Modeling and setup}

Consider a network of DC/AC power converters (e.g., islanded microgrid), each represented by a voltage phasor and interconnected via inductive transmission lines. We make the common assumption that  the system is in quasi-stationary state, i.e., around a nominal steady state frequency~$\omega^*$, see~\cite{kundur1994power, dorfler2016breaking}, meaning that all phasors are modeled with constant magnitude (1 per unit), and assume that the angle dynamics are controllable. For this, the converter dynamics are reduced to the following integrator dynamics,
\begin{align}
\label{eq: conv-dyn}
\dot\theta= u\,(\theta)+\omega^*\mathds{1}_n,\quad \theta(0)=\theta_0. 
\end{align}
Here, $u(\theta)=[u_1(\theta),\dots,u_n(\theta)]^\top\in\mathbb{R}^n$ is the control input, $\theta=[\theta_1,\dots,\theta_n]^\top\in\mathbb{R}^n$ is the vector of phase angles of the DC/AC converters and $\theta_0\in\mathbb{R}^n$ is the initial angle vector. 
While the modeling choice in this section ignores the internal dynamics of the converter, it enables the design of the optimal controller in a concise, closed-form due to its simplicity and mathematical tractability.
Later, Section~\ref{sec: sims} considers { a network of} detailed internal converter dynamics, { with lossy transmission lines}, descendent from first-order principles as in~\cite{jouini2020steady}, and discusses a practical implementation of the control scheme.

For the control design in~\eqref{eq: conv-dyn}, we consider a scenario where synchrophasor measurements with respect to a global frame of reference are available to each converter. This is a reasonable scenario for a future power grid, as PMU installation is becoming increasingly widespread~\cite{usman2019applications}.
We define the set of nominal phase angles, rotating at a synchronous frequency $\omega^*$, as $\theta^*(t)=\omega^*\mathds{1}_n t+\theta^*_0\in\real^n$, where $\theta^*_0=[\theta^*_{01},\dots,\theta^*_{0n}]^\top\in\real^n$ is the nominal initial angle vector. Let $\theta^*_{kj}=\theta^*_k-\theta^*_j$ define the nominal phase angle difference between neighboring converters~$(k,j)\in\mathcal{E}$. 
{ Assuming inductive (i.e. lossless) transmission lines,} the active power deviation from the nominal is given by, $$ P_{e,k}(\theta)-P^*_{e,k}=\sum\limits_{j\in\mathcal{N}_k} b_{kj}  \left(\sin(\theta_{kj})-\sin(\theta^*_{kj})\right),$$ where $P_{e,k}(\theta)$ is the electrical power injected into the network at the $k-$th converter and $P_{e,k}^*$ is the nominal power drawn from a DC source behind the $k-$th converter. 

\begin{remark}
\label{rem: freq-droop}
Recall that the control law, 
\begin{align}
\label{eq: freq-drp}
u_k(\theta)=-1/d_k\,\left(P_{e,k}(\theta)-P^*_{e,k}\right), \quad d_k>0,\,~ k=1,\dots, n,   \end{align}
results in the first-order frequency-droop control, that represents a prevalent approach for primary control in islanded microgrids. This, however, results in stationary frequency errors, which requires~\eqref{eq: freq-drp} to be augmented with a secondary control architecture, namely the automated generation control~\cite{dorfler2016breaking}.
\end{remark}

Following Remark~\ref{rem: freq-droop}, our goal in this work is to use measurements obtained from PMUs to synthesize a feedback controller with optimality guarantees. This will be shown to coincide with {\em the angular droop} control proposed in~\cite{6987381, 7419922}. This controller stabilizes the phase angle error (with respect to a nominal steady state angle) and is characterized by {\em zero} frequency deviation at stationarity.

\subsection{Optimal control problem formulation}
Consider the following optimization problem,
\begin{align}
\label{eq: main-prob}
\min_{u\in\real^n}& \int_0^\infty \sum_{k=1}^n \bigg(\alpha_k u_k^2(\theta)+ \\ &\frac{1}{4\alpha_k}\Big(\gamma_{k}(\theta_k-\theta_k^*)+  P_{e,k}(\theta)-P^*_{e,k}\Big)^2\bigg)\, \mathrm{d}t, \nonumber \\
&\text{s.t. } \dot{\theta} =u(\theta)+\omega^*\mathds{1}_n,\quad \theta(0)=\theta_0.\nonumber \end{align}
In~\eqref{eq: main-prob}, the first term in the running cost (the integrand) penalizes the control effort through the positive gains $\alpha_k>0,\, k=1,\dots, n$, by minimizing the scaled total power generation. { The second term is designed to accommodate a desired steady state behavior: power to angle droop, or $P-\theta$ droop, where $\gamma_k>0,\,~ k=1,\dots, n$, is a droop gain.} This droop behavior leads to zero stationary frequency error { and} can be seen as follows: 
under the optimal control $u^*(\theta)$ that solves~\eqref{eq: main-prob}, the running cost goes asymptotically to zero and it holds that, $$\lim\limits_{t\to \infty}\left(\gamma_{k}(\theta_k(t)-\theta_k^*(t))+  P_{e,k}(\theta)-P^*_{e,k}\right)=0.$$

More precisely, let $\theta^s_k:=\lim_{t\to\infty} \theta_k(t)$ be an induced steady state angle at the $k-$th converter. Then,
\begin{align}
\label{eq: powe-bal}
\gamma_{k}(\theta^s_k-\theta_k^*)&= P^*_{e,k}-P_{e,k}(\theta^s),\quad k=1,\dots, n.    \end{align}
Equation~\eqref{eq: powe-bal} describes the steady state as a power balance between the active power and angle deviation from the nominal value{ , where} $\theta^s=\{\theta^s_k\}_{k=1}^n$ given by \eqref{eq: powe-bal} is the induced steady state angle vector.
{ By taking the time derivative of~\eqref{eq: powe-bal}, we arrive at $\dot\theta^s_k=\omega^*$. I}t is evident that the steady state frequency error { is} zero. Intuitively,~\eqref{eq: powe-bal} is able to guarantee primary and secondary frequency control at once, i.e., resulting in a power system steady state with zero frequency error. In what follows, we synthesize an angle feedback control law $u^*(\theta)$ that uniquely solves~\eqref{eq: main-prob}.

\section{Inverse optimal control design}
\label{sec: optimal-control-design}
An innovative approach to optimal control synthesis was introduced in{~\cite{kalman1964linear, freeman1996inverse, haddad2011nonlinear, sepulchre2012constructive, 9429728}} and relies on the following idea: a feedback stabilizing control law associated with a control Lyapunov function for a dynamical system
is {\em first} determined and {\em then} a suitably chosen cost functional is found that satisfies the HJB equation. This constitutes the so-called {\em inverse} optimal control problem, where the running cost and the control parameters, representing a tuning knob, are determined {\em a posteriori}. This circumvents the need for an extensive search for a {\em good} cost functional { and} gives a value function from a suggested control Lyapunov function {\em for free} (without { analytically and} computationally expensive calculations). It also allows an easy control tuning with stability guarantees and is applicable to a wide range of optimal control problems. 

For our power network application, inverse optimal control { allows us to design} a {\em distributed} controller with feasible implementation.
In this section, we show that the optimization problem~\eqref{eq: main-prob} obeys the systematic { optimal} control synthesis presented in~\cite{freeman1996inverse, haddad2011nonlinear, sepulchre2012constructive, 9429728}. For convenience, we cite the following Theorem from our previous work~\cite{9429728}. The same results are also found in~\cite[Theorem 8.1]{freeman1996inverse},~\cite[Section 3.5]{sepulchre2012constructive}.

\begin{theorem}
\label{thm: H2-ctrl}
Consider the optimal control problem,
\begin{subequations}
\label{eq: opt-prob}
\begin{align}
\min_{u\in\real^n} & \int_0^\infty \norm{u(s)}^2_{\overline R} + \, q(x(s)) \; \dd s,\\
\mathrm{s.t.\quad } \dot x&= H^\top(x)\,u, \quad x(0)=x_0,
\label{eq: sys-dyn}
\end{align}
\end{subequations}
where $x,\;x_0\in\real^n,\; u\in\real^{n},\;  \overline R=\overline R^\top>0$, $q(x)$ is a function satisfying~$q(x)>0,\, q(0)=0$ and $H(x)\in\real^{m\times n}$ is the input matrix. Furthermore, let $V:\mathds{R}^n\mapsto \mathds{R}_{{ >} 0},$ be a continuously differentiable function associated with a feedback stabilizing control law, 
\begin{align}
\label{eq: H2-opt-cost-1}
u^*(x)&= -\frac{1}{2}\, {\overline R}^{-1} H(x)\, \nabla_x V,
\end{align}
where, $\nabla_x V^\top H^\top(x)\, u^*(x)<  -\norm{u^{*}(x)}^2_{\overline R}.$   
Define
\begin{align}
\label{eq: q-x}
q(x)=-\nabla_x V^\top H^\top(x) u^*(x) -\norm{u^{*}(x)}^2_{\overline R}.
\end{align} Then, the following statements hold: 
\begin{enumerate}
\item The unique optimal control is given by $u^*(x)$ in~\eqref{eq: H2-opt-cost-1}.
\item The optimal control problem~\eqref{eq: opt-prob} has the optimal value $V(x_0):=\inf\limits_{u\in\real^n} \int_0^\infty \norm{u(s)}^2_{\overline R} + \, q(x(s)) \; \dd s$ with $q(x)$ in~\eqref{eq: q-x}.
\end{enumerate}
\end{theorem} 
We make the following assumption.
\begin{assumption}
\label{ass: bounded-sol}
The { induced} steady state angle vector $\theta^{ s}=\{\theta^{ s}_k\}_{k=1}^n$ satisfies, 
$\mathcal{B}^\top \theta^{ s}\in\left(-\frac{\pi}{2}, \frac{\pi}{2}\right)^m$, where $\mathcal{B}\in\real^{n\times m}$ is the { incidence} matrix of the underlying graph $\mathcal{G}$.
\end{assumption}
Assumption~\ref{ass: bounded-sol} states that the difference { in steady state} voltage angles between neighboring nodes is not larger than $\pi/2$. This is commonly referred to as a {\em security constraint}~\cite{monshizadeh2017stability}. 
For ease of presentation, we introduce,
\begin{align*}
R&=\mathrm{diag}\{\alpha_1, \dots,\alpha_n\},\quad \Gamma=\mathrm{diag}\{\gamma_1, \dots,\gamma_n\}.
\end{align*}
Let the induced steady state angle $\theta^s$ be given by~\eqref{eq: powe-bal} and define the following function, that is used in deriving our main result,
\begin{align}
\label{eq: LF}
V(\theta)&=\frac{1}{2} \norm{\theta-\theta^s}^2_\Gamma\\
&+\sum_{k=1}^n \sum_{j\in\mathcal{N}_k}  b_{kj} \left(\cos(\theta_{kj})- \cos(\theta_{kj}^s)-(\theta_{kj}-\theta^s_{kj}) \sin(\theta^s_{kj})\right). \nonumber
\end{align}
Our main result is summarized in the following proposition.
\begin{proposition}
\label{prop: main-res}
Consider the optimal control problem~\eqref{eq: main-prob} under Assumption~\ref{ass: bounded-sol}.
Then, the following statements hold:
\begin{enumerate}
\item[i)] The optimal solution of~\eqref{eq: main-prob} at the $k-$th converter in a neighborhood of $\theta^s=\{\theta^s_k\}_{k=1}^n$ is the angular droop control defined as,
\begin{align}
\label{eq: opt-ctrl}
u_{k}^*(\theta)=-\frac{1}{2\alpha_k}\left(\gamma_k (\theta_k-\theta^*_k)+ P_{e,k}(\theta)-P^*_{e,k} \right).
\end{align}
\item[ii)] The steady state angle $\theta^s=\{\theta^s_k\}_{k=1}^n$ is locally asymptotically stable for the closed-loop system (i.e.,~\eqref{eq: conv-dyn} together with~\eqref{eq: opt-ctrl}).

\end{enumerate}
\end{proposition}
\begin{proof}
The proof relies on the observation that the optimal control problem~\eqref{eq: main-prob} satisfies the conditions of Theorem~\ref{thm: H2-ctrl} {\em locally}, i.e., in the vicinity of the induced steady state angle~$\theta^s$.

First, we establish the positive definiteness of the function $V$ around $\theta^s$. That is, we establish that $V(\theta^s)=0$ and $V(\theta)>0$ for $\theta\neq \theta^s$ with $\theta$ being in a neighborhood of $\theta^s$. For this, we follow a similar approach to~\cite{monshizadeh2017stability} and define $V_1(\theta)=\frac{1}{2} \norm{\theta-\theta^s}^2_\Gamma$ and $V_2(\theta)=W_2(\theta)-W_2(\theta^s)-(\theta-\theta^s)^\top \nabla_\theta W_2(\theta^s)$ with, $$W_2(\theta)=-\mathds{1}^\top_n\, \Xi\, \underline{\cos}(\mathcal{B}^\top \theta),$$
to rewrite the function $V(\theta)$ in~\eqref{eq: LF} as,
$V(\theta)= V_1(\theta)+ V_2(\theta).$

Note that $V_1$ is clearly positive definite around $\theta^s$. $V_2$ is positive definite around $\theta^s$ if $W_2$ is strictly convex around $\theta^s$. To show that $W_2$ is strictly convex around $\theta^s$, we introduce the coordinate change $\eta:= \mathcal{B}^\top \theta$ and calculate $\nabla^2_\eta W_2(\eta)=\Xi\; \underline{\cos}(\eta)$.
Under Assumption~\ref{ass: bounded-sol}, it holds that $\eta^s:=\mathcal{B}^\top \theta^s\in (-\frac{\pi}{2},\frac{\pi}{2})^m$ and hence $\nabla^2_\eta W_2(\eta)>0$, for $\eta$ in the neighborhood of $\eta^s$. This shows that $W_2(\eta)$ is strictly convex around $\eta^s$. Since strict convexity is invariant under affine maps, $W_2(\theta)$ is strictly convex around $\theta^s$. From the argumentation above, we deduce that $V_2$ and therefore $V$ is positive definite around $\theta^s$.

Second, we seek to apply Theorem~\ref{thm: H2-ctrl}. The gradient of $V(\theta)$ can be equivalently expressed as,
\begin{align}
\label{eq: gradient}
\nabla_\theta V&=\Gamma (\theta-\theta^s)+P_e(\theta)-P_e(\theta^s),  \\
&=\Gamma(\theta-\theta^*)+P_e(\theta)-P^*_e +\overbrace{\Gamma (\theta^*-\theta^s)+P^*_e-P_e(\theta^s)}^{=0}, \nonumber\\
&=\Gamma(\theta-\theta^*)+P_e(\theta)-P^*_e,\nonumber 
\end{align}
where $P_e(\theta)=[P_{e,1}(\theta),\dots,P_{e,n}(\theta)]^\top$, $P^*_e=[P^*_{e,1},\dots,P^*_{e,n}]^\top$ and the last term in the second step is zero by the induced steady state equation~\eqref{eq: powe-bal}.
This means that the control law~\eqref{eq: opt-ctrl} takes the form, 
$u^*(\theta)= -\frac{1}{2} R^{-1}\nabla_\theta V.$    
By left-multiplying with the gradient of $V$, it can be deduced that,
$$\dot V(\theta)=\nabla_\theta^\top V\, u^*(\theta)=-\frac{1}{2}\nabla_\theta^\top V\, R^{-1}\,\nabla_\theta V.$$
Denote by $\Omega$ a neighborhood of $\theta^s$. Note that $V$ is positive definite on $\Omega$ and $\dot V(\theta)\leq 0$ for all $\theta \in\Omega$. Let $S=\{\theta\in\Omega, \dot V(\theta)=0 \}$. The only trajectory that can stay in $S$ is where the gradient of $V$ given in \eqref{eq: gradient} vanishes, that is, only at $\theta=\theta^s$. By the Barbashin-Krasovskii theorem~\cite[Corollary 4.1]{K02}, the steady state angle $\theta^s$ is locally asymptotically stable.
Now, we write, $$\norm{u^*(\theta)}^2_R= \frac{1}{4}\nabla_\theta^\top V\, R^{-1}\,\nabla_\theta V.$$
Hence, for all $\theta\in\Omega$, $\nabla_\theta^\top V\, u^*(\theta)<-\norm{u^*(\theta)}^2_R.$
The cost functional~\eqref{eq: main-prob} can be compactly expressed as, $\int^\infty_0\norm{u(\theta)}^2_R+q(\theta)\, \dd s$,
with $$q(\theta)=-\nabla_\theta V^\top u^*(\theta) -\norm{u^{*}(\theta)}^2_R
=\frac{1}{4}\nabla_\theta V^\top R^{-1} \nabla_\theta V,$$ as given in~\eqref{eq: q-x} and explicitly written in~\eqref{eq: main-prob}. 

All in all, the control problem~\eqref{eq: main-prob} satisfies the conditions of Theorem~\ref{thm: H2-ctrl} {\em locally}, in a neighborhood of~$\theta^s$. It follows that~\eqref{eq: opt-ctrl} is an inverse optimal locally stabilizing control law for the system dynamics in~\eqref{eq: main-prob} and $V(\theta_0)$ in~\eqref{eq: LF} is the value function of~\eqref{eq: main-prob}. 
\end{proof}
The angular droop control~\eqref{eq: opt-ctrl} is {\em distributed}, i.e., it requires only knowledge of the neighboring angles $\theta_j, j\in\mathcal{N}_k, k\in\mathcal{V}$. Nonetheless, it can be implemented in a fully {\em decentralized} fashion by measuring the active power $P_{e,k}$ using PMUs. It is grid-forming according to definitions in~\cite{denis2017grid} and its tuning is easily understood: If the control gain $\alpha_k$ is { smaller}, more control effort is allowed { at the $k-$th converter}, and the rate of convergence towards an induced steady state angle $\theta^s$ is faster.
In this sense, the input matrix $R>0$ is a tuning knob that allows us to study combinations of the input penalty, while keeping the same value function.



\begin{remark}[LQR control]
{ Let $\Gamma=\mathrm{diag}\{\gamma_1, \dots,\gamma_n\}$, $R=\mathrm{diag}\{\alpha_1, \dots, \alpha_n\}$.} By linearizing { the cost functional~\eqref{eq: main-prob}} around $\theta=\theta^*$, { it} can be written as, 
	\begin{align}
	\label{eq: LQR}
	\int^\infty_0 u(s)^\top R \, u(s)+ \, (\theta(s)-\theta^*)^\top\, \overline Q\, \, (\theta(s)-\theta^*)\, \dd s,
	\end{align}
	where $\overline Q=\frac{1}{4} (\Gamma+ \mathcal{L})^\top R^{-1}(\Gamma+ \mathcal{L})$ and $\mathcal{L}=\mathcal{B}\, \Xi\, \mathcal{B}^\top$. Hence, the optimal control problem~\eqref{eq: main-prob} becomes an LQR problem~\cite{K02}. As delineated in~\cite{9429728}, after linearization around $\theta=\theta^*$, the control law~\eqref{eq: opt-ctrl} becomes, 
	\begin{align}
	    \label{eq: H2-ctrl}
	    u_{\text{LQR}}^*(\theta)=-\frac{1}{2}R^{-1} (\Gamma+\, \mathcal{L}) \,  (\theta-\theta^*),
	\end{align}
	and represents the $\mathcal{H}_2-$optimal controller of~\eqref{eq: LQR}.
\end{remark}

{ \section{Disturbance rejection and scalability: The linear case}
 \label{sec: distrubance-rej}
In this section, we follow the analysis in~\cite{tegling2017lcss,andreasson2017coherence} to compare the linearized angular droop controller~\eqref{eq: H2-ctrl} to standard frequency droop control from a transient performance perspective, that is, how well random disturbances are attenuated. In particular, we use the analysis framework from~\cite{tegling2017lcss,andreasson2017coherence} to demonstrate that the angle-based control~\eqref{eq:angdroopemma} can fundamentally improve the controller’s performance with respect to network size, and thereby its scalability to large networks.



\subsection{Disturbance attenuation and scalability: The linear case}
For this analysis, the closed-loop system dynamics are linearized around the desired steady state given by $\theta^*$, and the nominal frequency $\omega^*$. To simplify notation throughout this section, let the state vectors $\theta$ and $\omega$ represent deviations from nominal steady state. 

We assume that the system dynamics are subject to a disturbance $\eta= [\eta_1,\ldots,\eta_n]^\top$, which captures variations in generation and loads,  and into which we have also absorbed the constant power injections~$P^*$. The disturbance $\eta$ is modeled as a persistent stochastic variable, uncorrelated across converters. More precisely, we let $\eta$ be zero-mean white noise, such that $\mathbb{E}\{\eta(\tau)\eta^\top(t)\} = \delta(t-\tau)I_n$, where $\delta$ is the Dirac delta function. We refer the reader to~\cite{tegling2015price, tegling2017lcss} for more details on the disturbance model, as well as alternative input scenarios. 


Consider the linearized version of the angular droop controller~\eqref{eq: opt-ctrl} with, 
\begin{align}
\label{eq:angdroopemma}
\dot{\theta}&= -\frac{1}{2} R^{-1} (\Gamma+ L)\, (\theta-\theta^*)+\eta,
\end{align}
where $L=\mathcal{B}\, \Xi\, \mathcal{B}^\top$. 
On the other hand, the frequency droop control is given by,
\begin{equation}
\begin{aligned}
\begin{bmatrix}
\dot{\theta} \\ \dot{\omega} 
\end{bmatrix} = \begin{bmatrix}
0 & I_n \\ -M^{-1}\mathcal{L} & -M^{-1}D
\end{bmatrix}\begin{bmatrix}
\theta \\ \omega
\end{bmatrix} + \begin{bmatrix}
0\\ M^{-1}
\end{bmatrix} \eta, 
\end{aligned}
\label{eq:frequdroopemma}
\end{equation}
where $M$ and $D$ are diagonal matrices collecting all the inertia $m_i>0$ and damping coefficients $d_i>0$, respectively, with $i=1,\dots,n$.
Here, we have assumed both linearized power systems~\eqref{eq:angdroopemma} and~\eqref{eq:frequdroopemma} are subject to a disturbance input $\eta$.

\paragraph{Performance metric}
We evaluate the performance of the systems~\eqref{eq:angdroopemma} and~\eqref{eq:frequdroopemma} in terms of the following metric, given as \hn norm of an input-output system from the input $\eta$ to a suitably defined performance output. 
\begin{definition}[Angle coherence~\cite{tegling2017lcss}]
\label{def: angl-coh}
The angle coherence metric captures the steady-state variance of the converters' angle deviation from the network average, normalized by the network size $n$ and given by,
\begin{equation}
||\mathcal{S}||^2_{\mathrm{coh}} = \lim_{t \rightarrow\infty} \frac{1}{n}\sum_{i \in V} \mathbb{E}\left\lbrace \left( \theta_i(t) -  \overline\theta(t)  \right)^2 \right\rbrace,
\label{eq:perfmeas}
\end{equation} 
where $\overline \theta(t)=\frac{1}{n}\sum\limits_{i=1}^n \theta_i(t)$ is the average angle error.
\end{definition}

\vspace{1mm}

The performance metric in Definition~\ref{def: angl-coh} is given as the squared \hn norms of the systems~\eqref{eq:angdroopemma} and~\eqref{eq:frequdroopemma} with the performance output,
\begin{align*}
y_{\mathrm{coh}} & = \frac{1}{\sqrt{n}}\left(I_n - \frac{1}{n}\mathds{1}_n\mathds{1}_n^\top \right){\theta}.
\end{align*}

\paragraph{Comparison of angle and frequency droop}
We first make the following assumption for tractability purposes: 
\begin{assumption}
\label{ass:uniform}
Let the controller gains and parameters be uniform across all converters, i.e., $\alpha_i = \alpha$, $\gamma_i = \gamma$, $m_i = m$, and $d_i = d$ for all $i \in V$. 
\end{assumption}
Consider the following result: 
\begin{result}
\label{prop:perf}
Consider the linearized closed-loop dynamics first with the angular~\eqref{eq:angdroopemma} and second with the frequency droop~\eqref{eq:frequdroopemma} { under} Assumption~\ref{ass:uniform}. A comparative system performance is given in Table~\ref{tab:perfresults}.
\begin{table}[h!]
\caption{{Comparison} of linearized angle vs. frequency droop}
\begin{center}
 \setlength{\tabcolsep}{0.5em} 
 {\renewcommand{\arraystretch}{1.9}
\begin{tabular}{|c|c|c|}
\hline 
 &Angular droop~\eqref{eq:angdroopemma}& Frequency droop~\eqref{eq:frequdroopemma} \\ 
\hline 
Angle coherence & $||\mathcal{S}||^2_{\mathrm{coh}}= \frac{\alpha}{n}\sum\limits_{i = 2}^n \frac{1}{\gamma + \lambda_i}$ & $ ||\mathcal{S}||^2_{\mathrm{coh}}=\frac{1}{2\,d\,n}\sum\limits_{i = 2}^n \frac{1}{\lambda_i}$ \\ 
\hline 
\end{tabular}}
\end{center}
\label{tab:perfresults}
\end{table}

%
\end{result}
The proof follows that of~\cite[Lemma 1]{tegling2017lcss},~\cite{6580749} and is omitted here. 
From Table~\ref{tab:perfresults}, we make the following observations:
\begin{enumerate}
\item 
With angular droop, it is possible to state a uniform upper bound on the angle coherence. In particular, $||\mathcal{S}||^2_{\mathrm{coh}} < {\alpha}/{\gamma}$, which holds for {\em any} network size $n\in\mathds{N}$ and \emph{independently} of the graph topology. On the other hand, the performance for frequency droop, is proportional to $\frac{1}{n}\sum_{i = 2}^n\frac{1}{\lambda_i}$. This expression is  well-studied in the coherence literature, see e.g.,~\cite{Barooah2007, Bamieh2012}. In general, it cannot be uniformly bounded in $n$. Instead, it grows with $n$ for sparse network graphs, including, for example, tree graphs and graphs that can be embedded in two-dimensional lattices (e.g., planar graphs)~\cite{Barooah2007}. This leads to a performance degradation for large-scale networks. 
In summary, angular droop has fundamentally better scaling properties than frequency droop, leading to a better disturbance rejection for large, sparse graphs.  This is illustrated later through test case 2 in Section~\ref{sec: sims}. 


\item  We observe that for the angular droop~\eqref{eq:angdroopemma}, a small positive gain $\alpha$ minimizes the angle coherence. Similarly, increasing the damping gain $\gamma$ improves our performance metric. The droop gain $\gamma$ plays the same role for the angular droop as the gain $d$ for the frequency droop control. 
Thus, based on this $\mathcal{H}_2$ performance analysis we can select the control gains for an improved transient performance of the underlying power system model.



\end{enumerate}

\begin{remark}
The assumption on uniform controller parameters is only for mathematical tractability. It is, however, not important for the conclusion that the angle coherence is uniformly bounded for the angular droop~\eqref{eq:angdroopemma} only. For heterogeneous parameters, bounds can simply be stated in terms of the smallest and the largest gains. 
\end{remark}
}
\section{Implementation and numerical simulations}
\label{sec: sims}
{ Even though our previous analysis neglects the internal dynamics of each converter} in the { optimal} control synthesis, we propose a practical design of the angular droop control~\eqref{eq: opt-ctrl} for a network of high-order DC/AC converters. We numerically demonstrate in the next section that { high-order converter models can be accounted for}. 
\begin{figure}[ht!]
    \centering
    \includegraphics[width=.9\linewidth, trim= 0 4cm 0 2cm]{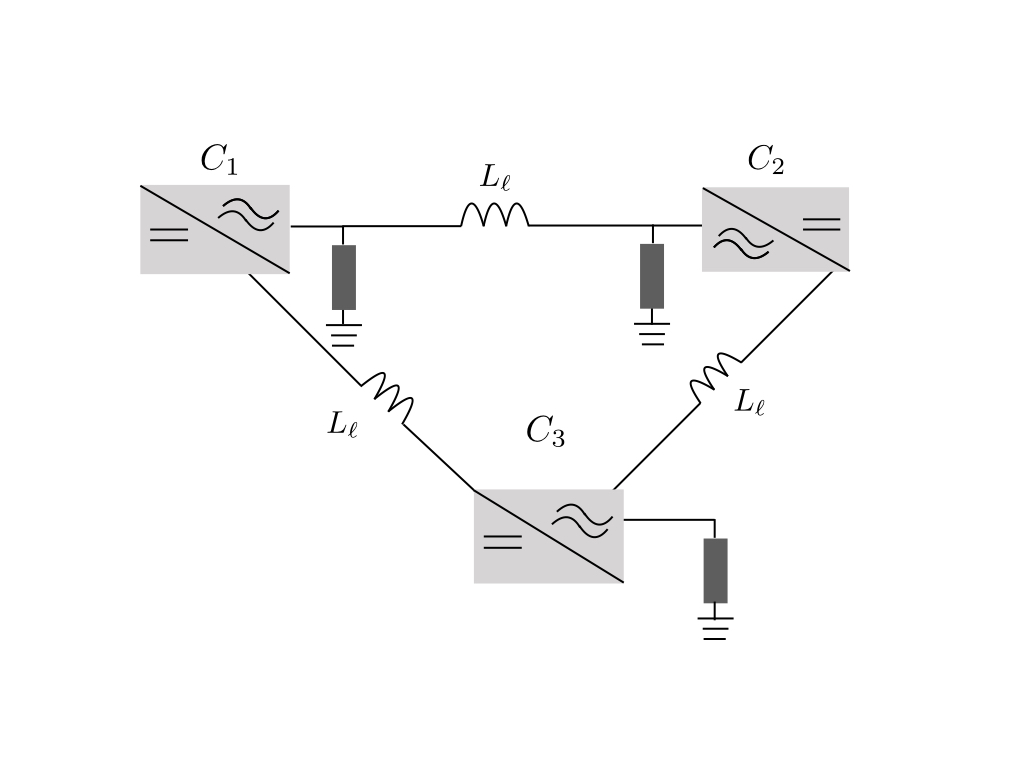}
    \caption{Three {high-order} DC/AC converter system described by the dynamics~\eqref{eq: open-loop} in closed-loop with angular droop~\eqref{eq: angle-droop}.}
    \label{fig: test-case-2-setup}
\end{figure}
{\subsection{Test case 1: Angular droop control}} 
For this, we consider the following three-phase averaged and balanced DC/AC converter dynamics in $\alpha\beta-$frame{~\cite{kundur1994power}}, adapted from~\cite{jouini2020steady},
\begin{equation}
\begin{aligned}
\label{eq: open-loop}
C_{dc} \dot v_{dc} &= -K_p\, (v_{dc}-v_{dc}^*\mathds{1}_n) - \frac{1}{2}U^\top i + i^*_{dc},  \\ 
L \,\dot i&= -R \,i + \frac{1}{2} U \,v_{dc} -v , \\
C \,\dot v&= -G\, v +i- \mathbf{B}\, i_{\ell},   \\
L_{\ell} \dot i_{\ell}&= -R_{\ell}\, i_{\ell} + \mathbf{B}^\top v,
\end{aligned}
\end{equation}
where the system parameters are summarized in Table~\ref{tab: parameters}. Note that the modulation signal $\overline u_k\in\real^{2}$, collected in the matrix $U$, represents the main input to the $k-$th DC/AC converter.

\begin{table}[h!]
\caption{Parameters of the multi converter system in Fig.~\ref{fig: test-case-2-setup} and~\ref{fig:summ-diag}.}
\begin{center}
	\resizebox{.5\textwidth}{!}{
			\begin{tabular}{lllc}
				\hline
				Symbol & Definition & Range & Value\\
				\hline
				$\overline u_k$  & modulation signal &$\real^{2}$&--\\
				$A$ & modulation amplitude & $[0 , 1]$& $0.33$\\
				$U=\text{diag}(\overline u_1, \dots, \overline u_n)$ & matrix of 
				input signals & $\real^{2n\times n}$ & --\\
				$v_{dc}^*$ & nominal DC voltage &$\real_{>0}$& 1000\\
				$i^*_{dc}$ & nominal DC current source &$\real^n$& $500\cdot 
				\mathds{1}_3$\\
				$C_{dc}$ & DC capacitance &$\real_{>0}$& $10^{-3}$\\
				$K_p$ & DC side control gain & 
				$\real_{>0}$& $0.5$ \\
				$R$ & AC { filter} resistance & $\real_{>0}$& $0.2$\\
				$L $ & AC { filter} inductance & $\real_{>0}$& $5\cdot 10^{-4}$\\
				$C $ & AC { filter} capacitance & $\real_{>0}$& $10^{-5}$\\
				$G $ & AC { filter} conductance & $\real_{>0}$& $0.1$\\ 
				\hline
				$L_{\ell}$ & line inductance & $\real_{>0}$& $5\cdot 10^{-5}$\\
				\hline
				$\alpha_k=\alpha, k=1\dots n$ & control gain  & $\real_{>0}$& $0.5$\\
				$\gamma_k=\gamma, k=1\dots n$ & droop gain & $\real_{>0}$& $10^6$\\
				$\mathbf{B}=I_2\otimes\mathcal{B}$ & extended incidence matrix & 
				$\real^{2n\times 2m}$&--\\
				\hline
				$v_{dc}=[v_{dc,1},\dots, v_{dc,n}]^\top$ &  DC capacitor voltage & 
				$\real^{n}$&-- \\
				$v=[v^\top_1,\dots, v^\top_n]^\top$ &  AC capacitor voltage & 
				$\real^{2n}$& --\\
				$i=[i^\top_1,\dots, i^\top_n]^\top$ &  AC inductance current & 
				$\real^{2n}$&--\\
				$i_{\ell}=[i^\top_{\ell,1}, \dots, i^\top_{\ell,m}]^\top$ & AC 
				line current & $\real^{2{ m}}$&--\\
				\hline
			\end{tabular}}
		\end{center}
		\label{tab: parameters}
	\end{table}

 

After introducing $i_{net}= \mathbf{B} i_{\ell}$ and defining the active power $\widehat P_{e,k}=v_k^\top i_{net,k}$, as well as the nominal steady state active power $\widehat P^*_{e,k}=v_k^{*\top} i_{net,k}^*$ at the $k-$th converter, we propose to implement the angular droop controller as follows,
\begin{align}
\label{eq: angle-droop}
\dot{\theta}_k &=-\frac{1}{2\alpha_k}\left(\gamma_k (\theta_k-\theta_k^*)+(\widehat P_{e,k}-\widehat P^*_{e,k}) \right)+\omega^*, \\ 
\overline u_k&=A\begin{bmatrix}
\cos(\theta_k) \\ \sin(\theta_k)   
\end{bmatrix},\nonumber
\end{align}
where $0<A<1$ is the amplitude of the control input. In Figure~\ref{fig:summ-diag}, we depict a summarizing block diagram of a single DC/AC converter whose system dynamics are given by~\eqref{eq: open-loop}, set in closed loop with the angular droop control~\eqref{eq: angle-droop}. Note that in this setup, the angular droop control~\eqref{eq: angle-droop} increments the converter internal dynamics with a virtual angle dynamics $\dot\theta_k$ that represents the phase angle of the modulation signal $\overline u_k$.

Next, we consider three DC/AC converters with open-loop dynamics described in~\eqref{eq: open-loop} in closed-loop with the angular droop control~\eqref{eq: angle-droop} as depicted in Figure~\ref{fig: test-case-2-setup}. The { desired} steady state angles are given (in rad) by $\theta^*_1(0)=0.951, \theta^*_2(0)=0.92, \theta_3^*(0)=0.967$, and thus satisfy Assumption \ref{ass: bounded-sol}. We select the control gains uniformly for all three converters with parameter values in Table~\ref{tab: parameters}. 
\begin{figure}[ht!]
    \centering
    \includegraphics[width=0.85\linewidth]{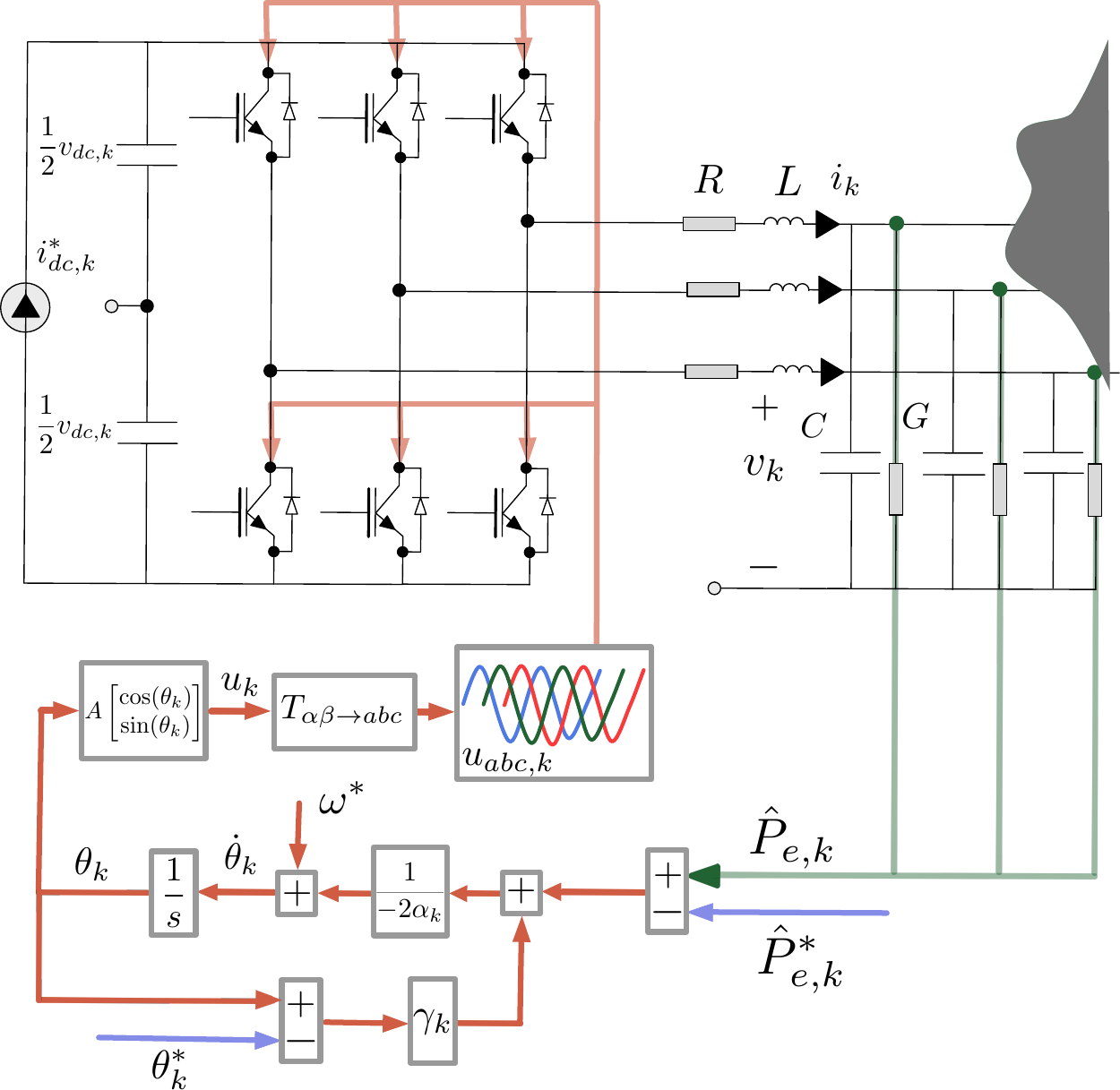} 
    \caption{Block diagram of the interconnection of a single three-phase balanced and averaged DC/AC converter with~\eqref{eq: open-loop} and~\eqref{eq: angle-droop}. The green arrows represent PMUs measurements. $T_{\alpha\beta \to abc}$ is the inverse of the Clark transformation~\cite{kundur1994power}.}
    \label{fig:summ-diag}
\end{figure}

We demonstrate the effectiveness of the proposed optimal controller both for angle stability and frequency synchonization via time-domain simulations before (under nominal conditions) and after an event corresponding to an increase in the load consumption at one of the converters. Fig.~\ref{fig: stability} illustrates angle stability for the initial angle values $\theta_1(0)=0.92, \theta_2(0)=0.90, \theta_3(0)=0.93$. We observe in simulations that a decrease in the gain $\alpha$ improves the angle transients, i.e., { it} results in faster convergence of the angles towards the induced steady state angle. Notice the first-order behavior of the phase angle trajectories dictated by~\eqref{eq: angle-droop}, while converging to their respective steady state values. Similarly, the frequencies synchronize at the nominal steady value $\omega^*=2\,\pi\, 50\, \text{rad/s}$. Fig.~\ref{fig: droop} illustrates the droop behavior in the phase angle after a sudden change in the load consumption and the corresponding effect on the frequency at the affected converter $(C_1)$. Note that the gain $\gamma>0$ defines the droop behavior between a sudden power change and the angle deviation at steady state. The angle drops correspond to peaks in the frequency time evolution, while the frequency error remains zero, also during the event.
\begin{figure}
    \centering
    \includegraphics[width=\linewidth]{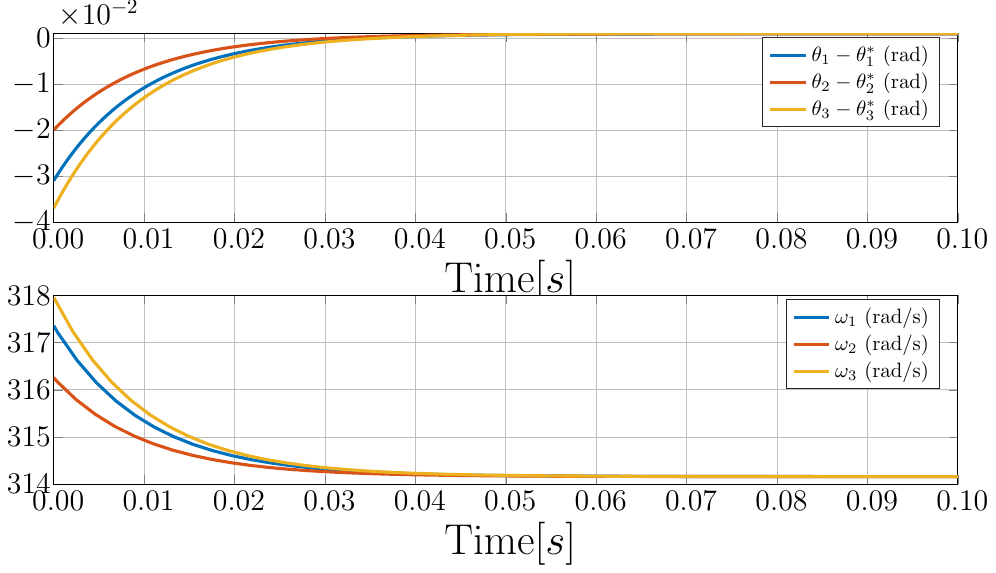}
    \caption{Time evolution of the converters' angle errors (in rad) with respect to the steady state $\theta^*$ initialized at $\theta_1(0)=0.92, \theta_2(0)=0.90, \theta_3(0)=0.93$ and frequency synchronization at $\omega^*=2\pi 50 \, \text{rad/s}$, for the setup in Fig.~\ref{fig: test-case-2-setup}.}
    \label{fig: stability}
\end{figure}
\begin{figure}
    \centering
    \includegraphics[width=\linewidth]{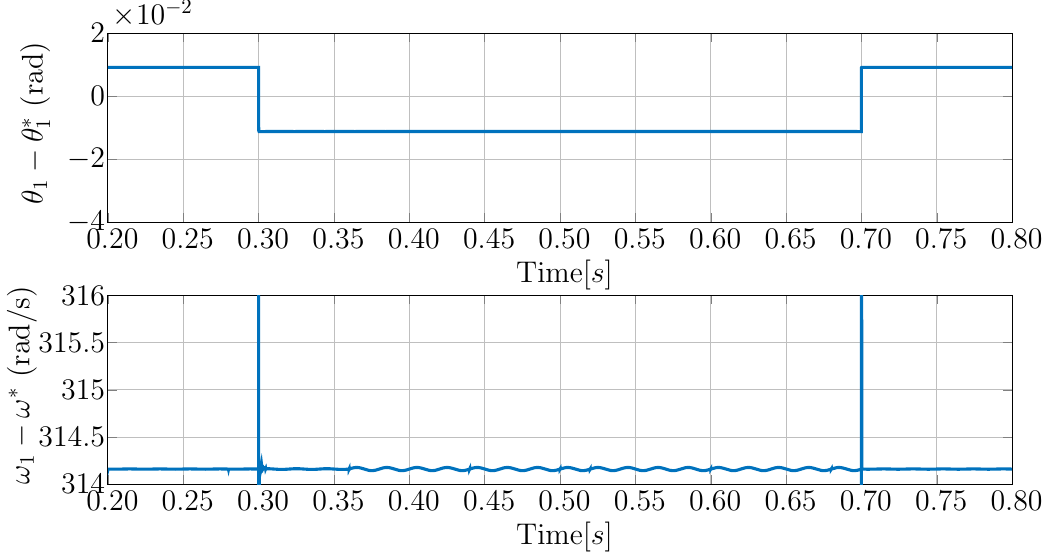}
    \caption{$P-\theta$ droop illustrated at the converter 1 (C1) angle and frequency after a sudden increase in the load consumption from $t=0.3$s to $t=0.7$s. The converter angle converges to the induced steady state angle $\theta_1^s$  during the load disturbance.}
    \label{fig: droop}
\end{figure}

Finally, we note that angular droop~\eqref{eq: opt-ctrl} has been numerically tested in~\cite{6987381,7419922} on different setups involving radial and loopy distribution systems. 

\subsection{ Test case 2: Comparison with frequency droop control} 
{ For the second test case, we compare qualitatively the transient performance { (see Definition 1 in~\cite{jouini2021inverse})} of angular and frequency droop after linearization, in a scalability analysis that is analogous to~\cite{andreasson2017coherence}. For this, consider the angular control~\eqref{eq: H2-ctrl} and frequency droop given by~\cite{andreasson2017coherence} with the same droop coefficients. We model two example path graph networks, first with 10 nodes and later with 100 nodes interconnected via inductive lines of unit susceptance (in p.u). { Then,} we subject the closed-loop dynamics to arbitrary initial angular perturbations.
The deviation of the angle error trajectories $\theta-\theta^*$ is depicted in Figure~\ref{fig: compare}. We observe that the convergence to a steady state is faster with the angular droop for both networks, i.e., a better transient performance (compare a) to c) and b) to d)). More importantly, however, we note that, as the network size grows from 10 to 100 nodes, the frequency droop shows a significantly degraded transient performance (compare d) to c)), while the angular droop shows similar transient performance for the larger network (in b)) as for the smaller one (in a)), and thus a better scalability. 

}

\begin{figure}
    \centering
    \includegraphics[width=\linewidth]{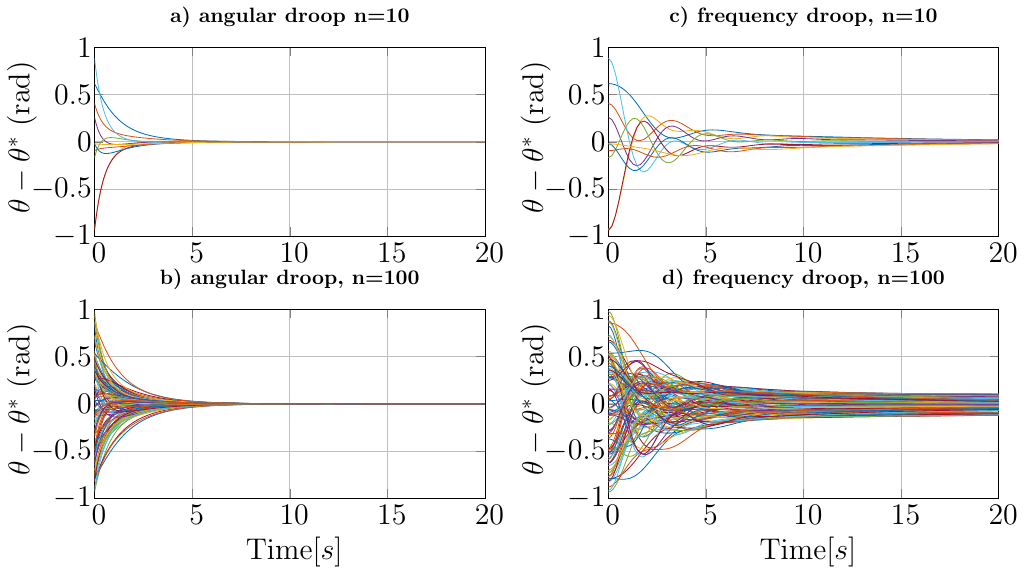}
    \caption{A comparison of the transient performance between the linearized angular droop \eqref{eq: H2-ctrl} displayed in a) and b) and the frequency droop~\cite{andreasson2017coherence} in c) and d) for a path network, where the network size increases from $n=10$ in a) and c) to $n=100$ nodes in b) and d).}
    \label{fig: compare}
\end{figure}

\section{Conclusion}
\label{sec: concl}
In this work, we proposed novel insights into the { design} of the angular droop control, that { establishes its optimality, while accounting} for phase angle stability with zero stationary frequency error. The angular droop control is distributed and showcases the utility of inverse optimal control theory in networked settings, and is numerically tested on power system simulations. 
It is of our future interest to study the stability of the angular droop control, while including internal DC/AC converter dynamics.




\bibliographystyle{ieeetran.bst}
\bibliography{root.bib}

\end{document}